\documentclass[11pt,twoside, reqno]{amsart}
\usepackage{amsfonts}
\usepackage{latexsym}
\usepackage{amssymb}
\usepackage{amsmath}
\usepackage{eqnarray}
\usepackage{enumerate}
\usepackage{dsfont}
\usepackage{multicol} 
\usepackage[pdftex,svgnames]{xcolor}
\usepackage{url}
\usepackage[utf8]{inputenc}
\usepackage{mathtools}
\usepackage[bookmarks=true,hyperindex,pdftex,colorlinks, citecolor=blue,linkcolor=blue,urlcolor=blue,linktocpage=true]{hyperref}
\usepackage{graphicx}
\graphicspath{{pictures/}}
\DeclareGraphicsExtensions{.pdf,.png,.jpg}

\theoremstyle{plain}
\newtheorem{theorem}{Theorem}[section]
\newtheorem{lemma}[theorem]{Lemma}

\theoremstyle{remark}

\theoremstyle{definition}
\newtheorem{definition}[theorem]{Definition}

\numberwithin{equation}{section}

\renewcommand{\ge}{\geqslant}

\newcommand{\eps}{\varepsilon}

\newcommand{\ifff}{if and only if }

\newcommand{\Real}{\operatorname{Re}}
\newcommand{\Slice}{\operatorname{Slice}}

\begin{document}
\title[DSD2P of Banach spaces is the same as the Daugavet property]{The diametral strong diameter 2 property of Banach spaces is the same as the Daugavet property}
\author{Vladimir Kadets}

\address{School of Mathematics and Informatics V.N. Karazin Kharkiv  National University,  61022 Kharkiv, Ukraine \newline
\href{http://orcid.org/0000-0002-5606-2679}{ORCID: \texttt{0000-0002-5606-2679} }}
\email{v.kateds@karazin.ua}

\thanks{ The research is done in frames of Ukrainian Ministry of Education and Science Research Program 0118U002036 and was partially supported  by project PGC2018-093794-B-I00 (MCIU/AEI/FEDER, UE)}

\begin{abstract}
We demonstrate the result stated in the title, thus answering an open question asked by  Julio Becerra Guerrero, Gin\'es L\'opez-P\'erez and Abraham Rueda Zoca in  J. Conv. Anal. \textbf{25}, no. 3 (2018). 
\end{abstract}

\maketitle

\section{Introduction} 

According to the famous theorem \cite{Daug} demonstrated by  Daugavet in 1963, the norm identity
\begin{equation} \label{eq1.1}
\|{\rm Id} - T\|=1+\|T\|,
\end{equation}
which has become known as the \emph{Daugavet equation}, holds for compact
operators on $C[0,1]$. 

Around 20 years ago the following general concept was introduced \cite{KSSW}: a Banach space $X$ has the \emph{Daugavet property}\/ if the  equation (\ref{eq1.1}) is fulfilled for every  rank one operator $T \in L(X)$ (here and below $L(X)$ denotes the space of all bounded linear operators $G \colon X \to X$). 

Surprisingly, the Daugavet property of $X$ implies the validity of  (\ref{eq1.1}) for much wider classes of operators than operators of rank~$1$: for example, for compact operators, weakly compact operators  \cite{KSSW}, operators that do not fix a copy of $\ell_1$ \cite{Shvid}, narrow operators \cite{KadSW2}, SCD-operators \cite{SCDsets}, etc.

Although the Daugavet property is of isometric nature, it has a number of linear-topological consequences. For example, a space with the Daugavet property cannot be reflexive, it cannot possess an unconditional basis, and so on. We refer to \cite[Section 1.4]{SpearsBook} for more motivation and the history of the subject.

Let $B_X$ be the unit ball of a Banach space $X$. A \emph{slice} of $B_X$  is a non-empty  part of $B_X$ that is cut out by a real hyperplane. Given $x^{\ast} \in X^{\ast}$ with $\|x^*\| = 1$ and $\alpha > 0$, denote the corresponding slice as
\[
\Slice(x^{\ast}, \alpha):= \left\{ x \in B_X \colon \Real{x^{\ast}(x)} > 1 - \alpha\right\}.
\]
The Daugavet property of $X$ can be reformulated in terms of slices of the unit ball: 
\begin{enumerate}[(i)]
\item  the property holds \ifff for every slice $S$ of $B_X$, every $\eps > 0$ and every $x \in S_X$ there is $s \in S$ with $\|x - s\| > 2-\eps$.
\end{enumerate}

There are two characterizations more \cite{KSSW, Shvid}, relevant to the subject of this paper, where slices are substituted by relatively weakly open subsets or convex combinations of slices, respectively:

\begin{enumerate}

\item[(ii)] $X$ has the Daugavet property  \ifff for every  non-empty  relatively weakly open subset $U$ of $B_X$, every $\eps > 0$ and every $x \in S_X$ there is $y \in U$ with $\|x - y\| > 2-\eps$.

\item[(iii)] $X$ has the Daugavet property  \ifff for every finite collection $S_1, \ldots, S_n$ of slices of $B_X$ and every collection of positive scalars $\lambda_1, \ldots, \lambda_n$ with $\lambda_1 + \ldots + \lambda_n = 1$ the corresponding convex combination $C = \lambda_1S_1 + \ldots + \lambda_nS_n$ of slices has the property that for  every $\eps > 0$ and every $x \in S_X$ there is $y \in C$ with $\|x - y\| > 2-\eps$.
\end{enumerate}

In \cite{IvaKad} it was asked whether the condition $x \in S_X$ in (i) can be changed to $x \in S \cap S_X$. In order to formalize this question the following definition was introduced:

\begin{definition} \label{defBadPr}
A Banach space $X$ is said to be \textit{a space with bad projections}  if for every slice $S$ of $B_X$, every $\eps > 0$ and every $x \in S \cap S_X$ there is $s \in S$ with $\|x - s\| > 2-\eps$. We denote this condition $X \in SBP$, or ``$X$ is an SBP space''.
\end{definition}
The name is motivated by the fact that $X \in SBP$ \ifff $\|{\rm Id} - P\| \ge 2$ for every projection $P \in L(X)$ of rank~$1$.

It was demonstrated \cite{IvaKad} that every absolute sum of two SBP spaces is an SBP space. On the other hand, the Daugavet property is inherited only by $\ell_1$ and $\ell_\infty$ sums, consequently there are SBP spaces that do not have the Daugavet property.

Motivated by this result,  Becerra Guerrero, L\'opez-P\'erez and  Rueda Zoca \cite{BLR18} introduced the following concepts.

\begin{definition}[{\cite[Definition 2.1]{BLR18}}] 
A Banach space $X$ is said to have the \textit{diametral diameter two property} (DD2P)  if for every non-empty relatively weakly open subset  $U$ of of $B_X$, every $\eps > 0$ and every $x \in U \cap S_X$ there is $y \in U$ with $\|x - y\| > 2-\eps$. 
\end{definition}

Comparing this with the characterization (ii), one easily sees that the Daugavet property implies the DD2P. On the other hand, DD2P is stable under all $\ell_p$ sums  \cite[Theorem 2.11]{BLR18}, so the inverse implication does not hold.

Finally, we are prepared for the main subject of interest of this paper.

\begin{definition}[{\cite[Definition 3.1]{BLR18}}] \label{defDiametral2}
A Banach space $X$ is said to have the \textit{diametral strong diameter two property} (DSD2P)  if for every convex combination  $C$ of slices of $B_X$, every $\delta > 0$ and every $x \in C$ there is $y \in C$ with $\|x - y\| > 1 + \|x\| -\delta$. 
\end{definition}

Again, from the characterization (iii) of  the Daugavet property, one can deduce \cite[Example 3.3]{BLR18} that the Daugavet property implies the DSD2P. 

The aim of this article is to demonstrate that the inverse implication is also true, so the diametral strong diameter 2 property of Banach spaces is the same as the Daugavet property.  The validity of this inverse implication was suggested in \cite[Question 4.1]{BLR18}. This question remained open since 2015, when the arXive preprint of \cite{BLR18} was published, and was mentioned as open problem in \cite{Hall2016, Hall, NadThes, PirkThes}. Our result, combined with the known results about the Daugavet property, gives also the positive answers to Questions 4.3 and 4.4 of \cite{BLR18} (the last one was already solved directly in \cite{Hall2016}).

\section{The main result}

Let us start with a useful geometrical concept.

\begin{definition}[{\cite[Definition 1.2.9]{KadThes}}] 
Let $X$ be a normed space,  $\eps > 0$ and $x,y \in X$.  The elements $x,y$ are said to be {\it $\eps$-quasi-codirected}   if  $\|x+y\| \ge \|x\|+\|y\| - \eps$. 
\end{definition}

\begin{lemma}[{\cite[Lemma 1.2.10]{KadThes}}] \label{LemEps-quasi}
Let  $x,y \in X$ be $\eps$-quasi-codirected. Then for every $a, b > 0$ the elements $ax, by$ are $(\eps \max\{a, b\})$-quasi-codirected.
\end{lemma}
\begin{proof}
Without loss of generality we may assume $a \ge b$. Then $a =\max\{a, b\}$ and
\begin{align*}
\|ax+by\| &= \|a(x+y) - (a - b)y\| \ge a\|x+y\| - (a - b)\|y\| \\
 &\ge a(\|x\|+\|y\| - \eps) - (a - b)\|y\| = a\|x\|+b\|y\| - a \eps.
\end{align*}
\end{proof}

\begin{theorem}  \label{Theo-main}
Let $X$ be a Banach space with the diametral strong diameter two property. Then 
$X$ has the  Daugavet property.
\end{theorem}
\begin{proof}
We are going to demonstrate that $X$ satisfies the characterization (i). 
Let  $\eps > 0$, $x \in S_X$ and a slice $S = \Slice(x^{\ast}, \alpha)$ of $B_X$ be given, where  $x^{\ast} \in X^{\ast}$, $\|x^*\| = 1$ and $\alpha > 0$. We need to demonstrate the existence of $s \in S$ with $\|x - s\| > 2 - \eps$. To this end, take $\beta = \min\{\alpha, \eps/2\}$ and consider  the slices 
$$S_1 = \Slice(x^{\ast}, \beta) \subset S,  \quad S_2 = \Slice(- x^{\ast}, \beta) = - S_1
$$ 
and the convex combination $C$ of slices defined as 
\begin{equation}\label{eq-setC}
C = \frac{1}{2}S_1 + \frac{1}{2}S_2 = \left\{\frac{1}{2}y_1 - \frac{1}{2}y_2 \colon y_1, y_2 \in S_1\right\}.
\end{equation}
Remark that every element $z \in S_1$ has $\|z\| \ge \Real x^\ast (z) > 1 - \beta \ge 1 - \eps/2$ and  $S_1$ has not empty norm-interior 
\[
 W = \left\{ z \in X \colon \|z\| < 1, \Real{x^{\ast}(z)} > 1 - \beta \right\}.
\]
Then $0 \in \frac{1}{2}W - \frac{1}{2}W \subset C$ is a norm-interior point of $C$, so there is such an $r \in (0, \frac{1}{2})$ that $r B_X \subset C$. Consider $r x \in C$ and $\delta \in \left(0, \eps/(2r)\right)$. According to Definition \ref{defDiametral2}, there is $y \in C$ with $\|rx - y\| > 1 + r\|x\| -\delta$. By \eqref{eq-setC}, we may represent $y$ as 
$y = \frac{1}{2}s - \frac{1}{2}h$
with $s, h \in S_1$.  Then
$$
\left\|rx - \frac{1}{2}s\right\| = \left\|rx - y + \frac{1}{2}h\right\| \ge \|rx - y\| - \frac{1}{2} > r\|x\| + \frac{1}{2}  - \delta.
$$
This means that the elements $rx$ and $-\frac{1}{2}s$ are $\delta$-quasi-codirected. Applying Lemma \ref{LemEps-quasi}, we obtain that the elements $x$ and $- s$ are $\frac{\delta}{r}$-quasi-codirected, that is 
$$
\|x - s\| \ge \|x\| +  \|s\| - \frac{\delta}{r} \ge 2 - \frac{\eps}{2} - \frac{\delta}{r} > 2 - \eps.
$$
Since $s \in S_1 \subset S$, this completes the proof.
\end{proof}

\end{document}